\documentclass[reqno]{amsart}
\usepackage{amssymb,amsmath}
\usepackage{amsthm}
\usepackage{color,graphicx}
\usepackage{hyperref}
\usepackage{verbatim}

\newcommand{\va}{\varphi}
\newcommand{\ka}{\mu}
\newcommand{\kp}{\kappa}
\newcommand{\ps}{\mu}
\newcommand{\Dt}{\mathcal{D}^{\theta}_\xi}

\newcommand{\de}{\delta} 
\newcommand{\D}{D_\xi^\theta} 
\newcommand{\h}{\mathcal H}
\newcommand{\ha}{\hat{\sigma}}
\newcommand{\pa}{\hat{\phi}}
\newcommand{\lan}{\langle \xi \rangle}
\newcommand{\lanx}{\langle x \rangle}
\newcommand{\lanxN}{\langle x \rangle_N}
\newcommand{\p}{\partial}
\newcommand{\les}{\lesssim}
\newcommand{\z}{\mathcal{Z}}
\newcommand{\R}{\mathbb R}

\newcommand{\Z}{\mathbb Z}

\newcommand{\s}{\text{sgn}(\xi)}

\numberwithin{equation}{section}

\newtheorem{theorem}{Theorem}[section]
\newtheorem{proposition}[theorem]{Proposition}

\newtheorem{lemma}[theorem]{Lemma}

\begin{document}
\vglue-1cm \hskip1cm
\title[The Benjamin equation]{On uniqueness results for the  Benjamin equation}



\author[A. Cunha]{Alysson Cunha}
\address{IME-Universidade Federal de Goi\'as (UFG), 131, 74001-970, Goi\^an\-ia-GO, Bra\-zil}
\email{alysson@ufg.br}






\begin{abstract}
We prove that the uniqueness results obtained in \cite{urrea} for the Benjamin equation, cannot be extended for any pair of non-vanishing solutions. On the other hand, we study uniqueness results of solutions of the Benjamin equation. With this purpose, we showed that for any solutions $u$ and $v$ defined in $\R\times [0,T]$, if there exists an open set $I\subset \R$ such that $u(\cdot,0)$ and $v(\cdot,0)$ agree in $I$, $\p_t u(\cdot,0)$ and $\p_t v(\cdot,0)$ agree in $I$, then $u\equiv v$. To finish, a better version of this uniqueness result is also established.
\end{abstract}

\maketitle

\section{Introduction}\label{introduction}

In this work, we study the initial-value problem (IVP) concerning the Benjamin equation
\begin{equation}\label{b}
\begin{cases}
u_{t}+\mathcal{H}\partial_{x}^{2}u+\p_x^3u+uu_{x}=0, \;\;x,t\in \R\\
u(0,x)=\phi(x),
\end{cases}
\end{equation}
where $u=u(t,x)$ is a real-valued function and $\mathcal{H}$ stands for the Hilbert transform defined as
\begin{equation*}
\begin{split}
\mathcal{H}f(x)&=\frac{1}{\pi} \lim_{\epsilon \downarrow 0}\int_{|y|\geq \epsilon}\frac{f(x-y)}{y}dy\\
               &=-i(\s \hat f (\xi))^{\vee}(x).
\end{split}
\end{equation*}



The integral-differential equation \eqref{b} is a mathematical model to describe a class of the intermediate waves in the stratified fluid. It was deduced by Benjamin \cite{B}, to study gravity-capillary surface waves of solitary type on deep water. He also obtained the following conservation laws

\begin{equation}\label{media}
I_1(u)=\int_{-\infty}^\infty u(t,x)dx,
\end{equation}

\begin{equation}\label{norm}
I_2(u)=\int_{-\infty}^\infty u^2(t,x)dx
\end{equation}

and

\begin{equation*}\label{I3}
I_3(u)=\int_{-\infty}^\infty\Bigg \{\frac12  (\p_x u(t,x))^2 -\frac12 u(t,x)\h \p_x u(t,x)-\frac13 u^3 (t,x)\Bigg\}dx.
\end{equation*}

The IVP \eqref{b} has been extensively studied with respect existence, uniqueness and regularity of solutions in the Sobolev space $H^s(\R)$. In \cite{l1} Linares, using the Fourier restriction method together with the conservation quantity \eqref{norm} obtained global well-posedness of the IVP \eqref{b} in $L^2(\R)$. In the last years, this result was improved by many authors. Kosovo et al. \cite{k} proved local well-posedness in $H^s(\R)$, where $s>-3/4$. Li and Wu \cite{lw}, using the I-method, showed global well-posedness of the Benjamin equation in $H^s(\R)$, for $s>-3/4$. Finally, the best result is due a Chen, Guo and Xiao \cite{cgx}, that extended the global well-posedness to $H^s(\R)$ for $s\geq -3/4$.

There are many works about uniqueness properties for dispersive models. We can mentioned that for IVPs associated to the k-generalized Korteweg–de Vries (k-gKdV) equation 

\begin{equation*}
\p_t u+\p_x^3 u + u^k \p_x u=0, \quad t,x\in \R,
\end{equation*}  

and for the semi-linear Schrödinger (NLS) equation 

\begin{equation*}
\p_t u=i(\nabla u+V(x,t)u), \quad  \mbox{in} \ \R^n\times [0,T],
\end{equation*} 

Escauriaza, Kenig, Ponce and Vega \cite{ekpv1, ekpv2} obtained uniqueness results for any pair of solutions $u_1$ and $u_2$ in a appropriate class functions. Also in this direction, in \cite{bim1}, for the Ostrovsky equation and in \cite{bim2}, for the Zakharov–Kuznetsov (ZK) equation, Bustamante et al., showed that if the difference of two sufficiently regular solutions $u_1$ and $u_2$ have a suitable decay in two different times, then they are equal. Recently, the results in \cite{bim2} was improved in \cite{cfl}.        

On the other hand, also there exist results on Benjamin equation in the weighted Sobolev spaces 

$$\z_{s,r}=H^s(\R)\cap L^2(x^{2r}dx) \quad \mbox{and} \quad \dot{\z}_{s,r}=\{f\in \z_{s,r}: \hat f (0)=0\},$$
where $s,r\geq 0$.
In such spaces, Urrea \cite{urrea}, studied well-posedness and unique continuation properties of solutions. Here it can be seen that the condition $s\geq 2r$ is necessary for solution $u$ of the Benjamin equation satisfies the persistence property in $\z_{s,r}$. Next, these results are put in order.

The first is about persistence property and well-posedness of solutions. 

\noindent {\bf Theorem A} {\em
\begin{itemize}
\item [(i)] Let $s\geq 1$, $r\in [0,s/2]$ and $r<5/2$. If $\phi \in \z_{s,r}$, then the solution $u(x, t)$ of the IVP \eqref{b} satisfies that $u\in C([0,\infty) : \z_{s,r})$. Furthermore if $s > 3/2$, $r\in [0, s/2]$ and $r < 5/2$, then the IVP \eqref{b} is globally well-posed in $\z_{s,r}$.
\item [(ii)] For $r\in [5/2, 7/2)$, $r\leq s/2$, the IVP \eqref{b} is GWP in $\dot{\z}_{s,r}$. 
\end{itemize}}
\medskip

The second is a necessary condition for the persistence of solution in $\z_{s,r}$, where $s\geq 5$ and $r\geq 5/2$.

\noindent {\bf Theorem B} {\em
Let $u\in C([0,T]:\z_{4,2})$ be a solution of the IVP \eqref{b}. If there exist two different times $t_1, t_2 \in [0,T]$ such that $u(\cdot,t_j)\in \z_{5,5/2}$, for $j=1,2$, then $\hat \phi(0)=0$ and therefore $u(\cdot, t)\in \dot{\z}_{4,2}$, for all $t\in [0,T]$.}
\medskip

In the following, was established the unique continuation property. 

\noindent {\bf Theorem C} {\em
Suppose that $u\in C([0, T ] :\dot{\z}_{7,7/2-})$ is a solution of the IVP \eqref{b}. If there exist three different times $t_1,t_2,t_3\in [0,T ]$ such that $u(\cdot, t_j)\in \dot{\z}_{7,7/2}$, for $j=1,2,3$, then $u(x,t)\equiv 0$.}
\medskip




\medskip


\noindent {\bf Theorem D} {\em
 Suppose that $u\in C(\R;\z_{7,7/2-})$ is a nonzero solution of the IVP \eqref{b}. If $u(0)=\phi\in \dot{\z}_{10,4}$ and
$$\int_{-\infty}^{\infty}x\phi(x)dx\neq 0,$$
then, there exists $t^*\neq 0\in \R$ such that $u(t^*)\in \z_{8,4}$.}
\medskip

As observed in \cite{urrea},  Theorem D shows us that the hypothesis of three different times in Theorem C cannot be reduced for two different times. Hence, we see that it's may expect to have uniqueness results with two different times on $\dot{\z}_{s,r}$, only for decay $r>4$. In addition, about the decay of solutions, we also can mention that the author in \cite{flores}, by using conservation laws of the Benjamin-Ono equation improved the decay results (see Theorem 2) in \cite{flp1}.

Next, we point out the main contributions of this paper. Our first result (see Theorem \ref{non0} below) is inspired by the work of Fonseca, Linares and Ponce \cite{flp1}. Here, the authors, for the Benjamin-Ono (BO) equation

\begin{equation*}
u_{t}+\mathcal{H}\partial_{x}^{2}u+uu_{x}=0, \;\;x,t\in \R,
\end{equation*}
shows that the uniqueness result obtained by Fonseca and Ponce in \cite{fp} cannot be extended to any pair of non-vanishing solutions of the BO equation.  




\begin{theorem}\label{non0}
Let $u$ and $v$ solutions of the IVP \eqref{b} with initial data $\phi$ and $\varphi$, respectively. Suppose that $\phi, \va \in \z_{9-,4}$ satisfies $\phi\neq \varphi$,
\begin{equation}\label{norma}
\|\phi\|=\|\va\|,
\end{equation}
\begin{equation}\label{v}
\int \phi(x)dx=\int \va(x)dx 
\end{equation}
and
\begin{equation}\label{v1}
\int x\phi(x)dx=\int x\va(x)dx. 
\end{equation}

 Then $u\not = v$ and for all $T>0$
 \begin{equation}\label{u-v}
  u-v \in L^\infty([-T,T];\z_{9-,4}).
 \end{equation}
\end{theorem}

The above theorem tell us that uniqueness results established in \cite{urrea} do not extend to any solutions $u\neq 0$ and $v\neq 0$ of the Benjamin equation. 

In what follows, by assuming higher-order decay in initial data, we improve the last theorem.
\begin{theorem}\label{non}
Let $\theta\in (0,1/2)$ and $u,v$ solutions of the IVP \eqref{b} with initial data $\phi$ and $\varphi$, respectively. Suppose that $\phi, \va \in \z_{9+2\theta,4+\theta}$ satisfies $\phi\neq \varphi$, \eqref{norma}, \eqref{v} and \eqref{v1}, above.

 Then $u\not = v$ and for all $T>0$
 \begin{equation}\label{u-v}
  u-v \in L^\infty([-T,T];\z_{9+2\theta,4+\theta}).
 \end{equation}
\end{theorem}

For the next two results, we show a version of the Theorems \ref{non0} and \ref{non} for initial data with more low regularity. 

\begin{theorem}\label{low4}

Let $u$ and $v$ solutions of the IVP \eqref{b} with initial data $\phi$ and $\varphi$, respectively. Suppose that $\phi, \va \in \z_{8,4}$ satisfies $\phi\neq \varphi$, \eqref{norma}, \eqref{v} and \eqref{v1}, above.

 Then $u\not = v$ and for all $T>0$
 \begin{equation}\label{u-v}
  u-v \in L^\infty([-T,T];\z_{8,4}).
 \end{equation}
\end{theorem}

\begin{theorem}\label{low4theta}

Let $\theta\in (0,1/2)$ and $u,v$ solutions of the IVP \eqref{b} with initial data $\phi$ and $\varphi$, respectively. Suppose that $\phi, \va \in \z_{8+2\theta,4+\theta}$ satisfies $\phi\neq \varphi$, \eqref{norma}, \eqref{v} and \eqref{v1}, above.

 Then $u\not = v$ and for all $T>0$
 \begin{equation}\label{u-v}
  u-v \in L^\infty([-T,T];\z_{8+2\theta,4+\theta}).
 \end{equation}

\end{theorem}

The next result corresponds to an improvement of the Theorem D, in the sense that, it requires  more higher decay on the initial data $\phi$. 
\begin{theorem}\label{2t}
Let $\theta\in (0,1/2)$. Suppose that $u\in C(\R;\z_{7,7/2-})$ is a solution of the IVP \eqref{b}. If $u(0)=\phi\in \dot{\z}_{9+2\theta,4+\theta}$ and 
\begin{equation*}
\int_{-\infty}^{\infty}x\phi(x)dx\neq 0,
\end{equation*}
then $$u(t^*)\in \z_{9+2\theta,4+\theta},$$ 
where
$$t^*=\frac{-4}{\|\phi\|^2}\int_{-\infty}^{\infty}x\phi(x)dx.$$
\end{theorem}

Note that the Theorem \eqref{2t} implies that we only may expect to have uniqueness results on two different times in $\dot{\z}_{s,r}$, for a decay $r\geq 9/2$.

Our results on uniqueness are inspired in the recent work of Kenig, Ponce and Vega \cite{kpv}. In this paper, the authors use techniques of Complex analysis to establish uniqueness results for the Benjamin Ono equation. 

In the following we present our first main result.

\begin{theorem}\label{uni1}
Let $u$ and $v$ be the solutions of the IVP \ref{b} for $(x,t)\in \R\times [0,T]$ such that
\begin{equation}\label{pontual1}
u,v\in C([0,T];H^s(\R))\cup C^1((0,T);H^{s-3}(\R)), \quad s>7/2.
\end{equation}
If there exist an open set $I\subset \R$ such that 
\begin{equation}\label{interval}
u(x,0)=v(x,0) \quad \mbox{and} \quad \p_t u(x,0)=\p_t v(x,0), \quad \mbox{for all} \quad x\in I,
\end{equation}
then
\begin{equation}
u\equiv v.
\end{equation}
In particular, if $u\equiv 0$ in $I\times \{0\}$ and $\p_t u(x,0)=0$, $\forall x\in I$, then
\begin{equation}
u\equiv 0.
\end{equation}
\end{theorem}

Note that the hypotheses of the above theorem are more general than from those in \cite[Theorem 1.1]{kpv}. Here we require conditions only on initial data and the time derivative at the origin.

The next result is an improvement of the above Theorem.

\begin{theorem}\label{uni2}
Let $u$ and $v$ be the solutions of the IVP \ref{b} for $(x,t)\in \R\times [0,T]$ such that
\begin{equation}\label{pontual1}
u,v\in C([0,T];H^s(\R))\cup C^1((0,T);H^{s-3}(\R)), \quad s>7/2.
\end{equation}
If there exists an open set $I\subset \R$, $0\in I$, such that
\begin{equation}\label{indata}
u(x,0)=v(x,0),\quad x\in I,
\end{equation}
and for each $N\in \Z^{+}$
\begin{equation}\label{intcon}
\int_{|x|\leq R}|\p_t u(x,0)-\p_t v(x,0)|^2 dx\leq c_N R^N, \quad \mbox{as} \quad R\downarrow 0,
\end{equation}
then
\begin{equation}
u(x,t)=v(x,t),\quad (x,t)\in \R\times [0,T].
\end{equation}
\end{theorem}

The rest of this paper is organized as follows. Section \ref{notation} contains some notation and preliminary estimates that will be useful in the proof of our results. In Section \ref{non1} we give the proof of Theorems \ref{non0}--\ref{low4theta}. In Section \ref{2t1} is proved the Theorem \ref{2t}. Finally, in Section \ref{uniqueness1} is presented the proof of Theorems \ref{uni1} and \ref{uni2}.

\section{Notation and Preliminaries}\label{notation}
Let us introduce the notation that it's being used in this paper. We write $a \les b$ if there exists a constant $c$ such that $a\leq c b$; by $a\sim b$ we mean that $a\les b$ and $b\les a$. The notation $a\les_{l} b$ indicate that the constant $c$ depends on parameter $l$. For the usual $L^2(\R)$ norm we will write only $\|\cdot\|$. The Fourier transform of $f$ is defined by
$$\hat{f}(\xi)=\int_{\R}e^{-i\xi x}f(x)dx.$$

Given any $s\in \R$, we denote the following operators, via their Fourier transform as follows

$$\widehat{J^s f}(\xi)=(1+\xi^2)^{s/2}\hat f (\xi) \quad \mbox{and} \quad \widehat{D^s f}(\xi)=|\xi|^s\hat f (\xi).$$

We also define the $L^2$ based Sobolev space $H^s=H^s(\R)$ by
$$\{f\in \mathcal{S}'(\R):\|\langle \xi \rangle^s \hat f\|<\infty\},$$
where $\mathcal{S}'(\R)$ is the space of tempered distributions and $\langle \xi \rangle:=(1+\xi^2)^{1/2}$.

For all $r\in \R$, we setting $L^2_r(\R)$ as being the set of all functions such that

$$\int (1+x^2)^r |f(x)|^2 dx<\infty.$$ 

If $s,r\in \R$ we denote the weighted Sobolev space (as we already mentioned) by
$$\z_{s,r}=H^s(\R)\cap L^2_r(\R),$$
 endowed with the norm
$$\|\cdot\|_{\z_{s,r}}^2=\|\cdot\|_{H^s}^2+\|\cdot\|_{L^2_r}^2.$$
We also write $\dot{\z}_{s,r}=\{f\in \z_{s,r}: \hat f (0)=0\}$.

For our estimates in weighted spaces, we setting the truncated weights $\langle x\rangle_N$, $N\in \Z^{+}$, which are given by
\begin{eqnarray*}
\langle x\rangle_N:=\left\{\begin{array} {lccc}
\langle x \rangle \ \mathrm{if} \  |x|\leq N,\\
2N \ \mathrm{if} \ |x|\geq 3N,
\end{array} \right.
\end{eqnarray*}
where $\langle x \rangle = (1+x^2)^{1/2}$. Also, we assume that $\langle x\rangle_N$
is smooth and non-decreasing in $|x|$ with $\langle x\rangle_N'(x)\leq 1,$ for any
$x\geq 0$, and there exists a constant $c$ independent of $N$ such that
$|\langle x\rangle_N''(x)|\leq c \partial_x^{2}\langle x \rangle.$

Finally, putting $\mu(\xi,t):=e^{it(\xi^3-\xi|\xi|)}$, follows that the group generated by the linear part of Benjamin equation can be written, via their Fourier transform as $$\widehat{U(t)\phi}(\xi,t)=\mu(\xi,t)\hat \phi(\xi).$$

Thus, the integral equation associated with the IVP \eqref{b} is given by
\begin{equation}\label{inte}
u(t)=U(t)\phi-\int_0^t U(t-\tau)\kappa(\tau)d\tau,
\end{equation} 
where $\kappa(\tau):=\frac{1}{2}\p_x u^2$.

To prove our main results we need of the following preliminaries results. The next proposition will useful in the proof of Theorems \ref{non0}--\ref{2t}. 
\begin{proposition}\label{jota}
Let $\delta,\nu>0$ such that $J^{\delta}f\in L^{2}(\R)$ and
$\langle x\rangle^\nu f
\in L^{2}(\R).$ Then for any
$\beta \in (0,1)$
\begin{equation*}
\|J^{\beta \delta}(\langle x\rangle^{(1-\beta)\nu}f)\|\leq c\|\langle x
\rangle^\nu f\|^{1-\beta}\|J^{\delta}f\|^{\beta}.
\end{equation*}
\end{proposition}
\begin{proof}
See Lemma 1 in \cite{fp}. 
\end{proof}

\begin{lemma}\label{boundhilbert}
Let $-1/2<\nu<1/2$, then the Hilbert transform $\mathcal{H}$ is a bounded operator in $L^{2}(\omega^\nu dx),$ i.e.
\begin{eqnarray*}
\|\omega^\nu \mathcal{H}f\| \leq c \|\omega^\nu f \|,
\end{eqnarray*}
where $\omega=|x|$ or $\omega=\lanxN$ and $c$ is a constant independent of $N$.
\end{lemma}
\begin{proof}
See \cite{fp} and references therein.
\end{proof}

The following result help us to estimate the $L^2$ norm of some terms that will appear in the proof of Theorems \ref{non} and \ref{2t}.

\begin{proposition}\label{intj}
For all $\theta\in(0,1)$ 
\begin{equation}\label{intj1}
\|J^{2\theta}(x^j \p_x^kf)\|\les \|J^{2(\theta+j)+k} f\|+\|\lanx^{\theta+j+\frac{k}{2}}f\|,
\end{equation}
and
\begin{equation}\label{intj2}
\||x|^\theta \p_x^k(x^j f)\|\les \|J^{2(4+\theta)}f\|+\|\lanx^{\frac{2(4+\theta)(j+\theta)}{2(4+\theta)-k}}f\|,
\end{equation}
where $j$ and $k$ are non-negative integers such that $k<2(4+\theta)$.
\end{proposition}
\begin{proof}
First, we will deal with \eqref{intj2}. We divide the proof in two cases.

\textbf{Case} $k\leq j$. 

For all $0\leq l\leq k$, its possible to apply Lemma \ref{jota} to obtain
\begin{equation}\label{desj}
\|J^l (\lanx^{j-k+l+\theta}f)\|\les \|J^{2(4+\theta)}f\|+\|\lanx^{\frac{2(4+\theta)(j-k+l+\theta)}{2(4+\theta)-l}}f\|.
\end{equation}
Thus, by using Leibniz rule for derivatives and \eqref{desj} it seems that
\begin{equation}
\begin{split}
\||x|^\theta \p_x^k(x^j f)\|\les& \sum_{l=0}^{k}\|\lanx^{j-k+l+\theta}\p_x^l f\|\\
\les& \sum_{l=0}^{k} \|J^l(\lanx^{j-k+l+\theta}f)\|\\
\les& \sum_{l=0}^{k}(\|J^{2(4+\theta)}f\|+\|\lanx^{\frac{2(4+\theta)(j-k+l+\theta)}{2(4+\theta)-l}}f\|)\\
\les& \|J^{2(4+\theta)}f\|+\|\lanx^{\frac{2(4+\theta)(j+\theta)}{2(4+\theta)-k}}f\|,
\end{split}
\end{equation}

where above we used the inequality $$\frac{2(4+\theta)(j-k+l+\theta)}{2(4+\theta)-l}\leq \frac{2(4+\theta)(j+\theta)}{2(4+\theta)-k}.$$

\textbf{Case} $k>j$. 

For any $0\leq l\leq j$, again using Lemma \ref{jota} 

\begin{equation}\label{desj1}
\|J^{k-j}(\lanx^{j-l+\theta}f)\|\les \|J^{2(4+\theta)}f\|+\|\lanx^{\frac{2(4+\theta)(j-l+\theta)}{2(4+\theta)-k+l}}f\|.
\end{equation}
Then, by Leinitz rule and \eqref{desj1} follows that
\begin{equation}
\begin{split}
\||x|^\theta \p_x^k(x^j f)\|\les& \sum_{l=0}^{j}\|\lanx^{j-l+\theta}\p_x^{k-j}f\|\\
\les& \sum_{l=0}^{j} \|J^{k-j}(\lanx^{j-l+\theta}f)\|\\
\les& \sum_{l=0}^{j} (\|J^{2(4+\theta)}f\|+\|\lanx^{\frac{2(4+\theta)(j-l+\theta)}{2(4+\theta)-k+l}}f\|)\\
\les& \|J^{2(4+\theta)}f\|+\|\lanx^{\frac{2(4+\theta)(j+\theta)}{2(4+\theta)-k}}f\|,
\end{split}
\end{equation}

where above its used the inequality

$$\frac{2(4+\theta)(j-l+\theta)}{2(4+\theta)-k+l}\leq \frac{2(4+\theta)(j+\theta)}{2(4+\theta)-k}.$$

With respect to inequality \eqref{intj1}, the Plancherel identity implies that

\begin{equation}
\begin{split}
\|J^{2\theta}(x^j \p_x^kf)\|&=\|\lan^{2\theta}\p_\xi^j (\xi^k \hat f)\|,
\end{split}
\end{equation}
hence, from here we deal by similar way to above.

This finishes the proof. 
\end{proof}

Next, we present a characterization of Sobolev spaces defined as $L^p_s:=(1-\Delta)^{-s/2}L^{p}(\R^n)$.  

\begin{theorem}\label{stein}
	Let $b\in (0,1)$ and $2n/(n+2b)<p<\infty.$ Then $f\in L^{p}_{b}(\R^{n})$ if and only if
	\begin{itemize}
		\item [a)] $f\in L^{p}(\R^{n}),$ 
		\item [b)]
		$\mathcal{D}^{b}f(x)={\displaystyle \left (
			\int_{\R^{n}}\frac{|f(x)-f(y)|^{2}}{|x-y|^{n+2b}}dy\right)^{1/2}} \in
		L^{p}(\R^{n}),$ with,
		\begin{equation}\label{equiv}
		\|f\|_{b,p}\equiv \|(1-\Delta)^{b/2}f\|_{p}=\|J^{b}f\|_{p}\simeq \|f\|_{p}+\|D^{b}f\|_{p}\simeq \|f\|_{p}+\|\mathcal{D}^{b}f\|_{p},
		\end{equation}
		where, for $s\in \R$,
		$D^{s}=(-\Delta)^{s/2}$.
	\end{itemize}
\end{theorem}
\begin{proof}
	See Theorem 1 in \cite{Stein}.
\end{proof}

The above operator $\mathcal{D}^b$ sometimes is named of Stein derivative of order $b$. The advantage in using it, is that is possible deduce a few useful pointwise estimates. 

Moreover, from the above theorem, part b), with $p=2$ and $b\in(0,1)$, it's possible to obtain the following product estimate
\begin{equation}\label{Leib}
\|\mathcal{D}^{b}(fg)\|_2 \leq \|f\mathcal{D}^{b}g\|_2 + \|g\mathcal{D}^{b}f\|_2.
\end{equation}

In what follows, the next two results, are useful to obtain the Lemma \ref{DF} below.

\begin{proposition}\label{Pontual}
Let $b\in (0,1)$. For any $t>0$,
\begin{equation}
\mathcal{D}^{b}(e^{-itx|x|})\leq c(t^{b/2}+t^{b}|x|^{b}).
\end{equation}
\end{proposition}
\begin{proof}
See Proposition 2 in  \cite{NahasPonce}. See also \cite{CunhaPastor}.
\end{proof}
\begin{lemma}\label{P}
Let $b\in (0,1)$. There exists a constant $C_b>0$ such that, for all $t>0$ and $x\in \R$,
$$
\mathcal{D}^{b}(e^{itx^3})\leq C_b\Big(t^{b/3}+t^{1/3+2b/9}+(t^{1/3+2b/3}+t^{2b/3})|x|^{2b}\Big).
$$
\end{lemma}
\begin{proof}
See Lemma 2.2 in \cite{bum}.
\end{proof}
The following result is a key ingredient to deal with the estimates in Theorems \ref{non} and \ref{2t}. 
\begin{lemma}\label{DF} For all $\theta \in (0,1)$ and $t\in (0,\infty)$
\begin{equation}
\|\Dt(\ka(\xi,t) \hat{f})\|\lesssim \rho(t)\big(\|J^{2\theta}f\|+\||x|^\theta f\|\big),
\end{equation}
where $\rho$ is increasing in $t$.
\end{lemma}
\begin{proof}
Using inequality \eqref{Leib}, Lemma \ref{P} and Proposition \ref{Pontual} we obtain 
\begin{equation*}
\begin{split}
\|\Dt(\ka(\xi,t) \hat{f})\|\lesssim& \ \|\Dt (e^{-it\xi|\xi|})e^{it\xi^3}\hat{f}\|+\|e^{-it\xi|\xi|}\Dt(e^{it\xi^3}\hat{f})\|\\
\lesssim& \ \|(t^{\frac{\theta}{2}}+t^{\theta}|\xi|^\theta)\hat{f}\|+\|\big(t^{\frac{\theta}{3}}+t^{\frac{1}{3}+\frac{2\theta}{9}}+(t^{\frac{1}{3}+\frac{2\theta}{3}}+t^{\frac{2\theta}{3}})|\xi|^{2\theta}\big)\hat{f}\|+\\
&+\|\Dt \hat{f}\|\\
\lesssim& \ \rho(t)\big(\|\hat f\|+\||\xi|^\theta\hat f\|+\||\xi|^{2\theta}\hat f\|+\|\Dt \hat f\|\big)\\
\les & \ \rho(t)\big(\|\lan^{2\theta}\hat f\|+\|\Dt \hat f\|\big).
\end{split}
\end{equation*}
Hence, the Plancherel identity and inequality \eqref{equiv} give us the desired result.
\end{proof}

The next Lemma will be useful in the proof of Theorem \ref{uni1}.
\begin{lemma}\label{hzero}
Let $f\in H^s(\R)$, $s>1/2$ be a real valued function. If there exists an open set $I\subset \R$ such that
\begin{equation}
f(x)=\h f(x)=0, \quad \forall x\in I,
\end{equation}
then $f\equiv 0$.
\end{lemma}
\begin{proof}
See Corollary 2.2 in \cite{kpv}.
\end{proof}

To proof the Theorem \ref{uni2} we made use of the following result.

\begin{lemma}\label{fzero}
Let $f\in L^2(\R)$ be a real valued function. If there exists an open set $I\subset \R$, $0\in I$, such that
\begin{equation}\label{f0}
f(x,0)=0, \quad x\in I,
\end{equation}
and for each $N\in \Z^{+}$
\begin{equation}
\int_{|x|\leq R}|\h f(x)|^2 dx\leq c_N R^N \quad \mbox{as} \quad R\downarrow 0,
\end{equation}
then
\begin{equation}\label{fx0}
f(x)=0, \quad x\in \R.
\end{equation}
\end{lemma}
\begin{proof}
See Lemma 4.1 in \cite{kpv}.
\end{proof}

\section{Proof of Theorems \ref{non0}--\ref{low4theta}}\label{non1}

This section is devoted to proof Theorems \ref{non0}--\ref{low4theta}. For Theorems \ref{non0} and\ref{non}, we use estimates on the norm of weighted spaces of the group associated with the linear part of the Benjamin equation. Such technique is also present in \cite{dBO} and \cite{gbozk}.

In Theorems \ref{low4} and \ref{low4theta}, we use the approach given by the authors in \cite{flp1}.

\begin{proof}[Proof of Theorem \ref{non0}]
Let $u$ and $v$ be the solutions of the IVP \eqref{b}, with initial data $\phi$ and $\varphi$, respectively. Suppose also $\phi\neq \varphi$, then by putting $\sigma:=\phi-\varphi$, $w:=u-v$ and $z:=\frac12 \p_x (u^2-v^2)$,  the integral equation \eqref{inte} gives us

\begin{equation}
\begin{split}\label{int}
w(t)=\ &U(t)\sigma-\int_0^t U(t-\tau) z(\tau)d\tau.
\end{split}
\end{equation}
We observe that for all $T>0$, the Theorem A implies that $u,v\in C([-T,T];\z_{9-,5/2-})$. Choosing $1/4<\epsilon<1/2$ and setting  $s=s_\epsilon=\frac{2+\epsilon}{\epsilon}$, it follows that

\begin{equation}\label{epsilon}
s<9 \quad \mbox{and} \quad 2+\epsilon<5/2.
\end{equation}

 As a consequence,  the constant 
\begin{equation}\label{Nuv}
 N:=\sup_{[-T,T]}\{\|u(t)\|_{\z_{s,2+\epsilon}}+\|v(t)\|_{\z_{s,2+\epsilon}}\},
\end{equation} 
   is finite.

Multiplying the last identity by $x^4$ and taking the Fourier transform follows that

\begin{equation}\label{2.8}
\begin{split}
 \p_\xi^4(\widehat{w(t)})=& \p_\xi^4(\mu(\xi,t)\ha)-\int_0^t  \p_\xi^4(\mu(\xi,t-\tau)\hat{z}(\tau))d\tau,
\end{split}
\end{equation}
we recall that $\mu(\xi,t)=e^{it(\xi^3-\xi|\xi|)}$.

Before to deal with the second right-hand of the last identity, we will need of the third derivative of the function $\ka \ha$, that, after several computations is given by

\begin{equation}
\begin{split}\label{der3}
\p_\xi^3 (\ka(\xi,t) \ha)=& \Big(\big(4it \de+6it-12t^2 \xi+54t^2|\xi|\xi-54t^2 \xi^3+8it^3|\xi|\xi^2-36it^3\xi^4+54it^3|\xi|\xi^4\\
&-27it^3\xi^6\big) \ha+\big(-6it\s +18it \xi-12t^2\xi^2 +36t^2|\xi|\xi^2-27t^2\xi^4\big)\p_\xi \ha\\&+\big(9it \xi^2-6it |\xi|\big)\p_\xi^2 \ha+\p_\xi^3 \ha\Big)\ka.
\end{split}
\end{equation}

Using the hypothesis \eqref{v} the term above that involves the delta Dirac function, can be computed as 

\begin{equation}\label{E1}
-4i t \ha \psi\delta
=-4i\Big(\int\sigma (x)dx\Big)\delta=-4i\Big(\int\phi(x)dx-\int \varphi(x)dx\Big)\delta=0.
\end{equation} 

Then, the delta Dirac function does not appear in identity \eqref{der3}. Using this fact, we can take the derivative with respect to $\xi$ variable in both sides of identity \eqref{der3}, to obtain                          

\begin{equation}
\begin{split}\label{der4}
\p_\xi^4 (\ka(\xi,t) \ha)=& \Big(\big(-12 t^2  -120i |\xi|+180 \xi^2-288 t^2 \xi^3-540i |\xi|\xi^3+324 \xi^5\\
&+48it^3 \xi |\xi|-216 t^2 \xi^6-216it|\xi|\xi^6+81 \xi^8+96i t^3 |\xi|\xi^4+16  t^4 \xi^4\big)\ha\\
&-4\big(-4i t \de+6-12 t^2 \xi-54i\xi|\xi|+54\xi^3-36t^2 \xi^4-54it |\xi|\xi^4+27\xi^6+\\
&+8it^3|\xi|\xi^2\big)\p_\xi \ha+6\big(-2it\s +6 \xi-4 t^2 \xi^2-12it|\xi|\xi^2 +9 \xi^4\big)\p_\xi^2\ha\\
&+4\big(3 \xi^2-2i t |\xi|\big)\p_\xi^3 \ha+\p_\xi^4 \ha\Big)\ka.
\end{split}
\end{equation}

The next step is to employ the identity \eqref{der4} to deal with  the first and second terms on the right-hand side of the identity \eqref{int}. First, we will compute explicitly the term in \eqref{der4} which includes the delta Dirac function. To do this, it's seen that the hypothesis \eqref{v1} implies

\begin{equation}\label{E14}
16i t \p_\xi \ha \psi\delta=
-16\Big(\int x \sigma (x)dx\Big) \delta=-16\Big(\int x \phi(x)dx-\int x\varphi(x)dx\Big) \delta=0.
\end{equation}
From \eqref{E1} and \eqref{E14}, to estimate the first term on the right-hand side of \eqref{2.8} we may write
\begin{equation}
\begin{split}\label{dtheta}
\|x^4 U(t)\sigma\|\les & \ \|\ha\|+\||\xi|\ha\|+\|\xi^2 \ha\|+\|\xi^3 \ha\|+\||\xi|\xi^3 \ha\|+\|\xi^5\ha\|+\|\|\xi|\xi \ha\|\\
&+\|\xi^6 \ha\|+\||\xi|\xi^6 \ha\|+\|\xi^8 \ha\|+\||\xi|\xi^4 \ha\|+\|\xi^4 \ha\|+\|\p_\xi\ha\|\\
&+\|\xi\p_\xi\ha\|+\| |\xi|\xi \p_\xi\ha\|+\|\xi^3 \p_\xi \ha\|+\|\xi^4 \p_\xi\ha\|+\| |\xi|\xi^4 \p_\xi\ha\|\\
&+\|\xi^6 \p_\xi\ha\|+\||\xi|^3 \p_\xi\ha\|+\|\p_\xi^2\ha\|+\|\xi \p_\xi^2\ha\|+\|\xi^2 \p_\xi^2\ha\|+\||\xi|^3 \p_\xi^2\ha\|\\
&+\|\xi^4 \p_\xi^2\ha\|+\|\xi^2\p_\xi^3\ha\|+\|\xi \p_\xi^3\ha\|+\|\p_\xi^4\ha\|\\
:=& \ B_1+\cdots+B_{28}.
\end{split}
\end{equation}

We will estimate only some terms in \eqref{dtheta}. Using Lemma \ref{jota} (with $\delta=4$, $\nu=8$) and Plancherel identity it follows that  
\begin{equation}
\begin{split}\label{E21}
\|B_{19}\|&\les \|\xi^6 \p_\xi \hat \sigma\|\\
&\les \|\lan^5 \hat \sigma\|+\|J_\xi(\lan^6 \hat \sigma)\|\\
&\les \|J^{5}\sigma \|+\|J_\xi^4 \hat \sigma\|+\|\lan^8 \hat \sigma\|\\
&\les \|J^{8}\sigma\|+\|\lanx^4\sigma\|.
\end{split}
\end{equation} 
Also by Proposition \ref{jota} (with $\delta=\nu=4$) we obtain

\begin{equation}
\begin{split}\label{E27}
\| B_{27}\|\les& \ \|\xi \p_\xi^3 \hat \sigma\|\\
\les& \ \|\hat \sigma\|+\|J_\xi(\lan \hat \sigma)\|+\|J^{3}_\xi(\lan \hat \sigma)\|\\
\les& \ \|J_\xi^4 \hat \sigma\|+\|\lan^4 \hat \sigma\|\\
\les& \  \|\lanx^4\sigma\|+\|J^{4}\sigma\|.
\end{split}
\end{equation}

To estimate the other terms we use the same argument above. Thus, proceeding in this way we conclude that 
\begin{equation}\label{Ej}
\| B_j\|\les \|J^{8}\sigma\|+\|\lanx^{4}\sigma\|, \quad j=1,...,28.
\end{equation}
Then, \eqref{E1}--\eqref{Ej} imply that

\begin{equation}
\begin{split}\label{Usigma}
\|x^{4}U(t)\sigma\|\les \|J^{8}\sigma\|+\|x^{4}\sigma\|.
\end{split}
\end{equation}

In the following we will estimate the second term on the right-hand side of \eqref{2.8}. First, observe that $\hat z(\tau,\xi)=\frac{i}{2}\xi(\widehat{u^2}-\widehat{v^2})$, then for any $\tau \in [-T,T]$
\begin{equation}\label{z0}
\hat{z}(\tau,0)=0,
\end{equation}
and
\begin{equation}\label{pz0}
\p_\xi \hat {z}(\tau,0)=\frac{i}{2}(u^2-v^2)^{\wedge}(\tau,0)=\frac{i}{2}(\|u(\tau)\|^2-\|v(\tau)\|^2)=\frac{i}{2}(\|\phi\|^2-\|\varphi\|^2)=0,
\end{equation}
where in the last identity we used the conservation law \eqref{norm} and \eqref{norma}.
Hence, from \eqref{z0}, \eqref{pz0} and putting $z(\tau)$ instead of $\sigma$, in \eqref{Usigma}, we obtain

\begin{equation}\label{zest}
\|x^4 U(t-\tau)z(\tau)\|\les \|J^{8}z(\tau)\|+\|x^{4}z(\tau)\|.
\end{equation}

Since
\begin{equation}
z=\frac12\big(\p_x w (u+v)+w \p_x(u+v)\big),
\end{equation}
by \eqref{zest} follows that for all $\tau \in [-T,T]$

\begin{equation}
\begin{split}\label{zx}
\|x^{4}U(t-\tau)z(\tau)\|\les& \ \|J^{8}(\p_x w (u+v))\|+\|J^{8}(w \p_x(u+v))\|\\
                    &+\|x^{4}w \p_x(u+v)\|+\|x^{4}\p_x w (u+v)\|.
\end{split}
\end{equation}

Next, we will estimate the terms on the right-hand side of the inequality \eqref{zx}. By using \eqref{Nuv}, Sobolev's embedding and Proposition \ref{jota} (with $\delta=s$ and $\nu=2+\epsilon$) we obtain 

\begin{equation}
\begin{split}\label{x4pxw}
\|x^{4}\p_x w (u+v)\|&\leq \|x^2(u+v)\|_{L^\infty_x}\|x^{2}\p_x w\|\\
&\les (\|x^2u\|_{L^\infty_x}+\|x^2v\|_{L^\infty_x})(\|xw\|+\|J(\lanx^{2}w)\|)\\
&\les (\|J(\lanx^2 u)\|+\|J(\lanx^2 v)\|)(\|J^{s}w\|+\|\lanx^{2+\epsilon}w\|)\\
&\les (\|J^s u\|+\|\lanx^{2+\epsilon} u\|+\|J^s v\|+\|\lanx^{2+\epsilon} v\|)^2\\
&\les N^2.
\end{split}
\end{equation}

Analogously to the last inequality we see that
\begin{equation}
\begin{split}\label{4theta}
\|x^{4}w \p_x(u+v)\|\les \ N^2.
\end{split}
\end{equation}

To estimate the fist term on the right-hand side of \eqref{zx}, since $H^8(\R)$ is a Banach algebra, follow that
\begin{equation}
\begin{split}\label{j120}
\|J^{8}(\p_x w (u+v))\|\les \|J^8(\p_x w)\|\|J^8(u+v)\|\les \|J^9(u-v)\|\|J^8(u+v)\|\les N^2.
\end{split}
\end{equation}

Also, in a similar way to \eqref{j120} it's seen that
\begin{equation}
\begin{split}\label{j123}
\|J^{8}(w \p_x(u+v))\|&\les N^2.
\end{split}
\end{equation}
Thus, from \eqref{zx}--\eqref{j123} we deduce that

\begin{equation}
\begin{split}\label{zpx}
\|x^4 U(t-\tau)z(\tau)\|\les N^2, \quad \mbox{for all} \ \tau\in [-T,T].
\end{split}
\end{equation}
Therefore, \eqref{int}, \eqref{Usigma} and \eqref{zpx} imply that

\begin{equation*}
\begin{split}
\|x^{4}w(t)\|&\les \|J^{8}\sigma\|+\|x^{4}\sigma\| +N^2\int_0^t d\tau\\
             &\les \|J^{8}\phi\|+\|x^{4}\phi\|+\|J^{8}\varphi\|+\|x^{4}\varphi\|+|t|N^2, \quad \mbox{for any} \ t\in [-T,T].
\end{split}
\end{equation*}

This gives us the desired result.

\end{proof}

\begin{proof}[Proof of Theorem \ref{non}]
Suppose that $u$ and $v$ are solutions of the IVP \eqref{b}, with $u(0)=\phi$, $v(0)=\varphi$, $\varphi\neq \phi$ and $\varphi,\phi\in \z_{9+2\theta,4+\theta}$. Here we are assuming  the same notation as at beginning of the proof of Theorem \ref{non0}, that is $\sigma:=\phi-\varphi$, $w:=u-v$ and $z:=\frac12 \p_x (u^2-v^2)$. Using Theorem A, we see that for all $T>0$, $u,v\in C([-T,T];\z_{9+2\theta,5/2-})$. In addition, Theorem \ref{non0} implies that $w\in C([-T,T];L^2(|x|^8 dx))$.  Hence, the constant 
\begin{equation}\label{Muv}
M:=\sup_{[0,T]}\{\|u(t)\|_{\z_{9+2\theta,2+\theta}}+\|v(t)\|_{\z_{9+2\theta,2+\theta}}+\|x^4 w(t)\|\},
\end{equation} 
 is finite. 

 Then, proceeding by similar way as the proof of Theorem \ref{non0}, we obtain

\begin{equation*}
\begin{split}
 \p_\xi^4(\widehat{w(t)})=& \p_\xi^4(\mu(\xi,t)\ha)-\int_0^t  \p_\xi^4(\mu(\xi,t-\tau)\hat{z}(\tau))d\tau.
\end{split}
\end{equation*}

From the last identity we conclude that

\begin{equation}\label{Dxi}
\begin{split}
D_\xi^\theta \p_\xi^4(\widehat{w(t)})=&D_\xi^\theta \p_\xi^4(\mu(\xi,t)\ha)-\int_0^t D_\xi^\theta \p_\xi^4(\mu(\xi,t-\tau)\hat{z}(\tau))d\tau.
\end{split}
\end{equation} 

In the next, we will estimate the terms in the right-hand side of identity \eqref{Dxi}. To do this, we observe that from the hypothesis \eqref{v} and \eqref{v1}, follows that the identity \eqref{der4} still hold, that is

\begin{equation*}
\begin{split}
\p_\xi^4 (\ka(\xi,t) \ha)=& \Big(\big(-12 t^2  -120i |\xi|+180 \xi^2-288 t^2 \xi^3-540i |\xi|\xi^3+324 \xi^5\\
&+48it^3 \xi |\xi|-216 t^2 \xi^6-216it|\xi|\xi^6+81 \xi^8+96i t^3 |\xi|\xi^4+16  t^4 \xi^4\big)\ha\\
&-4\big(6-12 t^2 \xi-54i\xi|\xi|+54\xi^3-36t^2 \xi^4-54it |\xi|\xi^4+27\xi^6+\\
&+8it^3|\xi|\xi^2\big)\p_\xi \ha+6\big(-2it\s +6 \xi-4 t^2 \xi^2-12it|\xi|\xi^2 +9 \xi^4\big)\p_\xi^2\ha\\
&+4\big(3 \xi^2-2i t |\xi|\big)\p_\xi^3 \ha+\p_\xi^4 \ha\Big)\ka,
\end{split}
\end{equation*}
where above, the delta Dirac function does not appear.

Using Plancherel identity and the last equality, the first term on the right-hand side of \eqref{Dxi} can be estimated as follows
\begin{equation}
\begin{split}\label{dtheta1}
\||x|^{4+\theta}U(t)\sigma\|\les & \ \|\D(\ps \ha)\|+\|\D(\ps |\xi|\ha)\|+\|\D(\ps \xi^2 \ha)\|+\|\D(\ps \xi^3 \ha)\|+\\
&+\|\D(\ps |\xi|\xi^3 \ha)\|+\|\D(\ps \xi^5\ha)\|+\|\D(\ps |\xi|\xi \ha)\|+\|\D(\ps \xi^6 \ha)\|\\
&+\|\D(\ps |\xi|\xi^6 \ha)\|+\|\D(\ps \xi^8 \ha)\|+\|\D(\ps|\xi|\xi^4 \ha)\|\\
&+\|\D(\ps \xi^4 \ha)\|+\|\D(\ps \p_\xi\ha)\|+\|\D(\ps \xi\p_\xi\ha)\|+\|\D(\ps |\xi|\xi \p_\xi\ha)\|\\
&+\|\D(\ps\xi^3 \p_\xi \ha)\|+\|\D(\ps\xi^4 \p_\xi\ha)\|+\|\D(\ps |\xi|\xi^4 \p_\xi\ha)\|\\
&+\|\D(\ps\xi^6 \p_\xi\ha)\|+\|\D(\ps|\xi|\xi^2 \p_\xi\ha)\|+\|\D(\ps \s \p_\xi^2\ha)\|\\
&+\|\D(\ps\xi \p_\xi^2\ha)\|+\|\D(\ps\xi^2 \p_\xi^2\ha)\|+\|\D(\ps|\xi|\xi^2 \p_\xi^2\ha)\|\\
&+\|\D(\ps\xi^4 \p_\xi^2\ha)\|+\|\D(\ps\xi^2\p_\xi^3\ha)\|+\|\D(\ps|\xi| \p_\xi^3\ha)\|\\
&+\|\D(\ps\p_\xi^4\ha)\|\\
:=& \ E_1+\cdots+E_{28}.
\end{split}
\end{equation}

Next, we will give details of the estimate of some terms in \eqref{dtheta1}. From Lemma \ref{DF}, Proposition \ref{boundhilbert} (with $\omega=|x|^\theta$)  and Proposition \ref{intj} (with $j=2$, $k=0$) it follows that  
\begin{equation*}
\begin{split}\label{E21}
\|E_{21}\|&\les \|J^{2\theta}\h (x^2 \sigma)\|+\||x|^\theta \h(x^2 \sigma)\|\\
&\les \|J^{2\theta}(x^2 \sigma)\|+\||x|^\theta x^2 \sigma\|\\
&\les \|J^{2(\theta+2)}\sigma\|+\|\lanx^{2+\theta}\sigma\|\\
&\les \|J^{2(4+\theta)}\sigma\|+\||x|^{4+\theta}\sigma\|.
\end{split}
\end{equation*} 
Also by Lemmas \ref{DF} and Proposition \ref{intj} (with $j=3$, $k=2$) we obtain

\begin{equation*}
\begin{split}\label{E271}
\| E_{26}\|\les& \ \|J^{2\theta}\p_x^2(x^3 \sigma)\|+\||x|^\theta \p_x^2(x^3 \sigma)\|\\
\les&\  \|J^{2\theta}(x\sigma)\|+\|J^{2\theta} (x^2\p_x \sigma)\|+\|J^{2\theta}(x^3 \p_x^2 \sigma)\|+\||x|^\theta \p_x^2(x^3 \sigma)\|\\
\les&\  \|J^{2(\theta+1)}\sigma\|+\|\lanx^{\theta+1}\sigma\|+\|J^{2(\theta+2)+1}\sigma\|+\|\lanx^{\theta+2+\frac12}\sigma\|+\\
&+\|J^{2(\theta+3)+2}\|+\|\lanx^{\theta+4}\sigma\|+\|\lanx^{\frac{2(4+\theta)(3+\theta)}{2(4+\theta)-2}}\sigma\|\\
\les& \ \|J^{2(4+\theta)}\sigma\|+\||x|^{4+\theta}\sigma\|.
\end{split}
\end{equation*}

To deal with the other terms we can proceed in a similar way to above. Thus, we can deduce that 
\begin{equation}\label{Ej1}
\| E_j\|\les \|J^{2(4+\theta)}\sigma\|+\||x|^{4+\theta}\sigma\|, \quad j=1,...,28.
\end{equation}
Then, gathering togheter \eqref{dtheta1} and \eqref{Ej1}, we see that  

\begin{equation}
\begin{split}\label{Usigma1}
\||x|^{4+\theta}U(t)\sigma\|\les \|J^{2(4+\theta)}\sigma\|+\||x|^{4+\theta}\sigma\|.
\end{split}
\end{equation}

Now, we will estimate the second term on the right-hand side of \eqref{Dxi}. First we recall that $\hat{z}(\tau,0)=0$, for all $\tau\in [-T,T]$. Moreover, by analogous way to proof of Theorem \ref{non0}, the conservation law \eqref{norm} and the identity \eqref{norma} imply that $\p_\xi \hat {z}(\tau,0)=0$, for any $\tau\in [-T,T]$. As a consequence, putting $z(\tau)$ instead of $\sigma$, in \eqref{Usigma}, we obtain

\begin{equation}\label{zest1}
\||x|^{4+\theta}U(t-\tau)z(\tau)\|\les \|J^{2(4+\theta)}z(\tau)\|+\||x|^{4+\theta}z(\tau)\|.
\end{equation}

In view of
\begin{equation*}
z=\frac12\big(\p_x w (u+v)+w \p_x(u+v)\big),
\end{equation*}
by \eqref{zest} we can write, for all $\tau \in [-T,T]$

\begin{equation}
\begin{split}\label{zx1}
\||x|^{4+\theta}U(t-\tau)z(\tau)\|\les& \ \|J^{2(4+\theta)}(\p_x w (u+v))\|+\|J^{2(4+\theta)}(w \p_x(u+v))\|\\
                    &+\||x|^{4+\theta}w \p_x(u+v)\|+\||x|^{4+\theta}\p_x w (u+v)\|.
\end{split}
\end{equation}

In what follows, we will deal with terms on the right-hand side of the inequality \eqref{zx}. 

In view of definition \eqref{Muv}, the Sobolev's embedding and Proposition \ref{jota}(with $\delta=\nu=2+\theta$ and $\delta=\nu=4$) imply that 

\begin{equation}
\begin{split}
\||x|^{4+\theta}\p_x w (u+v)\|&\leq \||x|^{1+\theta}(u+v)\|_{L^\infty_x}\|x^{3}\p_x w\|\\
&\les (\|J(\lanx^{1+\theta} (u+v))\|)(\|x^2 w\|+\|J(\lanx^{3}w)\|)\\
&\les (\|J^{2+\theta} (u+v)\|+\|\lanx^{2+\theta} (u+v)\|)(\|J^{4}w\|+\|\lanx^{4}w\|)\\
&\les M^2.
\end{split}
\end{equation}
Moreover, proceeding as above
\begin{equation}
\begin{split}\label{4theta1}
\||x|^{4+\theta}w \p_x(u+v)\|\leq \||x|^\theta \p_x (u+v)\|_{L^\infty_x}\|x^{4}w\|\les M^2.
\end{split}
\end{equation}

In a similar way to \eqref{j120}, the first term on the right-hand side of \eqref{zx} can be estimated as follows 
\begin{equation}
\begin{split}\label{j1201}
\|J^{2(4+\theta)}(\p_x w (u+v))\| \les M^2.
\end{split}
\end{equation}

From way analogous to \eqref{j1201} we also deduce that
\begin{equation}
\begin{split}\label{j1231}
\|J^{2(4+\theta)}(w \p_x(u+v))\|&\les M^2.
\end{split}
\end{equation}
Finally, the inequalities \eqref{zx1}--\eqref{j1231} imply that, for all $\tau\in [-T,T]$

\begin{equation}
\begin{split}\label{zpx1}
\||x|^{4+\theta}U(t-\tau)z(\tau)\|\les M^2.
\end{split}
\end{equation}
Therefore, by Plancherel identity, \eqref{Dxi}, \eqref{Usigma1} and \eqref{zpx1} we obtain

\begin{equation*}
\begin{split}
\||x|^{4+\theta}w(t)\|&\les \ \|J^{2(4+\theta)}\sigma\|+\||x|^{4+\theta}\sigma\| +|t|M^2\\
&\les \ \|J^{2(4+\theta)}\phi\|+\|J^{2(4+\theta)}\varphi\|+\||x|^{4+\theta}\phi\|+\||x|^{4+\theta}\varphi\| +|t|M^2.
\end{split}
\end{equation*}
By the last inequality we conclude the proof of Theorem.
\end{proof}

\begin{proof}[Proof of Theorem \ref{low4}]

Follows similar ideas contained in \cite{flp1}.

\end{proof}

\begin{proof}[Proof of Theorem \ref{low4theta}]

Assume that $\theta\in (0,1/2)$. Let $u$ and $v$ solutions of the IVP \eqref{b} with initial data $\phi$ and $\varphi$, respectively. As before, let $\sigma=\phi-\varphi$ and  $w=u-v$. Hence, we see that $w$ satisfies the following equation

\begin{equation}\label{bw}
w_{t}+\mathcal{H}\partial_{x}^{2}w+\p_x^3 w +u\p_x w+\p_x v w=0, \;\;x,t\in \R.
\end{equation}

The global well-posedness in Sobolev spaces yield us $u,v\in C([-T,T];H^{8+2\theta})$, for all $T>0$. Moreover, the Theorem \ref{low4} implies that $w\in L^{\infty}([-T,T];\z_{8+2\theta,4})$. 
Thus, the constant 

\begin{equation*}
M_1:=\sup_{[-T,T]}\{\|u(t)\|_{H^{8+2\theta}}+\|v(t)\|_{H^{8+2\theta}}+\|w(t)\|_{\z_{8+2\theta,4}}\},
\end{equation*}
is finite.

We observe that by \eqref{media} and \eqref{norm}, the solution $w$ of \eqref{bw} satisfies the following conservation laws

\begin{equation}\label{law1}
\int w(x,t)dx=\int \sigma(x)dx
\end{equation}
and
\begin{equation}\label{law2}
\frac{d}{dt}\int x w(x,t)dx=\frac12(\|\phi\|^2-\|\varphi\|^2).
\end{equation}

Then, from \eqref{law1}, \eqref{law2} and the hypothesis \eqref{norma}--\eqref{v1} it follows that 

\begin{equation}\label{consw}
\int w(x,t)dx=0
\end{equation}

and

\begin{equation}\label{consxw}
\int x w(x,t)dx=0,
\end{equation}

for all $t$ in which the solution there exists.

Multiplying \eqref{bw} by $\lanxN^{2+2\theta}x^6 w$ and integrating on $\R$ we obtain

\begin{equation}\label{intwn}
\begin{split}
\frac12 \|\lanxN^{1+\theta}x^3 w\|^2+\int \lanxN^{1+\theta}x^3\h \p_x^2 w\lanxN^{1+\theta}&x^3 w +\int \lanxN^{2+2\theta}x^6w\p_x^3 w\\                                   
+\int \lanxN^{2+2\theta}&x^6(u\p_x w+\p_x v w)=0.
\end{split}
\end{equation}

Next, we will estimate the second term on the left-hand side of \eqref{intwn}.

It's seen that the identities

\begin{equation*}
x\h \p_x^2 w=\h \p_x^2(xw)-2\h \p_x w,
\end{equation*} 

\begin{equation*}
x^2\h \p_x^2 w=\h \p_x^2(x^2w)-4\h \p_x (xw)+2\h w
\end{equation*} 
and \eqref{consw} imply that

\begin{equation*}
x^3 \h \p_x^2 w=\h\p_x^2(x^3 w)-6\h \p_x (x^2 w)-2\h (xw).
\end{equation*}

Thus, 

\begin{equation}
\begin{split}
\lanxN^{1+\theta}x^3 \h \p_x^2 w&=\lanxN^{1+\theta}\h\p_x^2(x^3 w)-6\lanxN^{1+\theta}\h \p_x (x^2 w)-2\lanxN^{1+\theta}\h (xw)\\
&:= \ \mathcal A+\mathcal B+\mathcal C.
\end{split}
\end{equation}

Next, we will deal with each one of the terms above.

We can write

\begin{equation}\label{ma}
\begin{split}
\mathcal A&=[\lanxN^{1+\theta};\h]\p_x^2(x^3 w)+\h(\lanxN^{1+\theta}\p_x^2(x^3 w))\\
&=\  \mathcal A_1 +\h\p_x^2 (\lanxN^{1+\theta}x^3 w)-2\h (\p_x \lanxN^{1+\theta}\p_x(x^3 w))-\h(\p_x^2\lanxN^{1+\theta}x^3 w)\\
&= \ \mathcal A_1+\cdots+\mathcal A_4.
\end{split}
\end{equation}

Thus, by the Calder\'on commutator estimate (see Theorem 6 in \cite{fp} and references therein) we get

\begin{equation}\label{ma1}
\mathcal A_1\les \|\p_x^2\lanxN^{1+\theta}\|_{L^\infty}\|x^3 w\|\les M_1.
\end{equation}

Also, the fact that $\h$ is a bounded operator in $L^2$ implies that

\begin{equation}\label{ma4}
\mathcal A_4\les \|\p_x^2\lanxN^{1+\theta}x^3 w\|\les \|\p_x^2\lanxN^{1+\theta}\|_{L^\infty}\|x^3 w\|\les M_1.
\end{equation}

By returning the term $\mathcal A_2$ in \eqref{intwn} and using Plancherel identity we see that

\begin{equation}\label{ma2}
\int \h \p_x^2 (\lanxN^{1+\theta}x^3w)\lanxN^{1+\theta}x^3 w=0.
\end{equation}
To estimate $\mathcal A_3$, the inequality $|x\p_x \lanxN|\les \lanxN $ yields us 

\begin{equation}\label{ma3}
\begin{split}
\|\mathcal A_3\| &\les \ \|\lanxN^\theta \p_x \lanxN x^2 w\|+\|\lanxN^\theta \p_x \lanxN x^3 \p_x w\|\\
        &\les \ \|\lanxN^{1+\theta}x w\|+\underbrace{\|\lanxN^{1+\theta}x^2 \p_x w\|}_{P}\\
        &\les \ M_1+P.
\end{split}
\end{equation}

Using inequality $\lanxN \les 1+|x|$, Proposition \ref{boundhilbert} (with $\omega=\lanxN$) and  identity $\widehat{\p_x (x^2 w)}(0,t)=0$, the term $B$ can be estimated as follows

\begin{equation}
\begin{split}
\mathcal B&\les \ \|\lanxN^\theta \h \p_x(x^2 w)\|+\|\lanxN^\theta x \h \p_x(x^2 w)\|\\
 &\les \ \|\lanxN^{\theta}\p_x(x^2 w)\|+\|\lanxN^\theta \h(x \p_x(x^2 w))\|\\
 &\les \ \|\lanxN^\theta x w\|+\|\lanxN^\theta x^2 \p_x w\|+\|\lanxN^\theta x^2 w\|+\|\lanxN^\theta x^3 \p_x w\|\\
 &\les \ M_1+\|\lanxN^{1+\theta}x^{3}w\|+\underbrace{\|\lanxN^\theta \lanx^{3} \p_x w\|}_{F}.
\end{split}
\end{equation}

 The third term on the right-hand side of the last inequality can be estimated as

\begin{equation}
\begin{split}
F&\les \ \|J(\lanxN^\theta \lanx^3 w)\|+\|x^3 w\|+\|\lanxN^\theta \lanx^2 w\|\\
 &\les \ \|J^{1+\theta}(\lanx^3 w)\|+\|\lanxN^{1+\theta}\lanx^3 w\|\\
 &\les \ \|J^{4(1+\theta)} w\|+\|\lanx^4 w\|+\|\lanxN^{1+\theta}\lanx^3 w\|\\
 &\les \ M_1+\|\lanxN^{1+\theta} x^3 w\|,
\end{split}
\end{equation}

where above its used Lemma \ref{jota} (with $\gamma=\nu=1+\theta$ and $\gamma=4(1+\theta)$, $\nu=4$).

Also by the Proposition \ref{boundhilbert} and equality \eqref{consxw} we get

\begin{equation}
\begin{split}
\mathcal C& \les \ \|\lanxN^\theta\h (xw)\|+\|\lanxN^\theta x \h (xw)\|\\
  &\les \ \|\lanxN^\theta x w\|+\|\lanxN^\theta \h (x^2 w)\|\\
  &\les \ \|\lanxN^\theta x w\|+\|\lanxN^\theta x^2 w\| \\
  &\les \ \|\lanxN^{1+\theta}\lanx^3 w\|\\
  &\les  \ M_1+\|\lanxN^{1+\theta} x^3 w\|.
\end{split}
\end{equation}

In what follows we deal with the other terms in \eqref{intwn}.

An application of integration by parts gives us

\begin{equation}
\begin{split}
\int \lanxN^{2+2\theta}x^6w\p_x^3 w&=-\frac12\int \p_x^2(x^6 \lanxN^{2+2\theta})w\p_xw-\frac32\int \p_x(x^6 \lanxN^{2+2\theta})w\p_x^2w\\
&\les \ \|x^3 \lanxN^{1+\theta}w\|\|x\lanxN^{1+\theta}\p_x w\|+\underbrace{\|\lanx^2 \lanxN^{1+\theta}\p_x^2 w\|}_{L}\|x^3 \lanxN^{1+\theta}w\|\\
&\les \ \|x^3 \lanxN^{1+\theta}w\|(F+L),
\end{split}
\end{equation}
where above we also used the inequalities 

$$|\p_x(x^6 \lanxN^{2+2\theta})|\les |x|^5 \lanxN^{2+2\theta} \quad \mbox{and} \quad |\p_x^2(x^6 \lanxN^{2+2\theta})|\les x^4 \lanxN^{2+2\theta}.$$

To estimate $L$, we observe that by the equality

$$\p_x^2 (\lanx^2 \lanxN^{1+\theta}w)=\p_x^2(\lanx^2 \lanxN^{1+\theta})w+2\p_x(\lanx^2 \lanxN^{1+\theta})\p_x w+\lanx^2 \lanxN^{1+\theta}\p_x^2 w$$ 
and Lemma \ref{jota} it follows that

\begin{equation}
\begin{split}
 L &\les \|J^2(\lanx^2 \lanxN^{1+\theta}w)\|+M_1\\
 &\les \|J^6(\lanxN^{1+\theta}w)\|+\|\lanx^3 \lanxN^{1+\theta}w\|+M_1\\
 &\les \|J^{8+2\theta}w\|+\|\lanxN^{4+\theta}w\|+\|x^3 \lanxN^{1+\theta}w\|+M_1\\
 &\les \|x^3 \lanxN^{1+\theta}w\|+M_1,
\end{split}
\end{equation}
where above, we used Lemma \ref{jota} moreover the inequality $\lanxN^3\les (1+x^3)$.

In a similar way to the term $L$, we conclude that the term $P$ (in \eqref{ma3}) satisfies

\begin{equation}\label{P}
P\les \|x^3 \lanxN^{1+\theta}w\|+M_1.
\end{equation}

Hence, by \eqref{ma}--\eqref{ma3}

\begin{equation}
\mathcal A \les \|x^3 \lanxN^{1+\theta}w\|+M_1.
\end{equation}

About the fourth term in \eqref{intwn}, integration by parts, Sobolev's embedding and the inequality 
 $|\p_x(x^6 \lanxN^{2+2\theta})|\les (1+x^6)\lanxN^{2+2\theta}$  imply that

\begin{equation}
\begin{split}
\int \lanxN^{2+2\theta}x^6w(u\p_x w+\p_x v w)=&-\frac12\int \big(\p_x(x^6 \lanxN^{2+2\theta})u+x^6 \lanxN^{2+2\theta}\p_x u\big)w^2+\\
&+\int x^6 \lanxN^{2+2\theta}w^2 \p_x v\\
\les&  \ (\|u\|_{L^\infty}+\|\p_x v\|_{L^\infty})\|x^3 \lanxN^{1+\theta}w\|^2+M_1^2\\
\les& \ M_1(M_1+\|x^3 \lanxN^{1+\theta}w\|^2).
\end{split}
\end{equation}

Gathering together the above inequalities we obtain

\begin{equation*}
\frac{d}{dt}\|\lanxN^{1+\theta}x^3 w\|^2 \les (1+M_1)\|\lanxN^{1+\theta}x^3 w\|^2+M_1^2.
\end{equation*}

Using Gronwall's Lemma we conclude that

\begin{equation*}
\sup_{[-T,T]}\|\lanxN^{1+\theta}x^3 w\|\leq c(T).
\end{equation*}

Therefore, from the last inequality we conclude the proof.

\end{proof}

\section{Proof of Theorem \ref{2t}}\label{2t1}

\begin{proof}[Proof of Theorem \ref{2t}]

First of all, as we already mentioned the IVP \eqref{b} is GWP in $\dot{\z}_{s,r}$, where $5/2\leq r<7/2$ and $s\geq 2r$. Hence, from the hypothesis $\phi \in \dot{\z}_{9+2\theta,4+\theta}$ it follows that $u\in C(\R;\z_{9+2\theta,7/2-})$. 

We recall that the integral equation associated with IVP \eqref{b} is given by

\begin{equation}
\begin{split}\label{int1}
u(t)=\ &U(t)\phi-\frac12\int_0^t U(t-\tau) (\p_x u^2)(\tau)d\tau.
\end{split}
\end{equation}

Then, \eqref{int1} and Plancherel's identity imply that

\begin{equation}\label{2.9}
\begin{split}
 \p_\xi^4(\widehat{u(t)})=& \p_\xi^4(\mu(\xi,t)\hat \phi)-\int_0^t  \p_\xi^4(\mu(\xi,t-\tau)\hat{\kappa}(\tau))d\tau,
\end{split}
\end{equation} 

where $\kappa:=\frac12 \p_x u^2 $.

Since $\phi$ has zero mean value, we can use the identity \eqref{der4} to write 

\begin{equation}
\begin{split}\label{der4phi}
\p_\xi^4 (\ka(\xi,t) \pa)=& \Big(\big(-12 t^2  -120i |\xi|+180 \xi^2-288 t^2 \xi^3-540i |\xi|\xi^3+324 \xi^5\\
&+48it^3 \xi |\xi|-216 t^2 \xi^6-216it|\xi|\xi^6+81 \xi^8+96i t^3 |\xi|\xi^4+16  t^4 \xi^4\big)\pa\\
&-4\big(-4i t \de+6-12 t^2 \xi-54i\xi|\xi|+54\xi^3-36t^2 \xi^4-54it |\xi|\xi^4+27\xi^6+\\
&+8it^3|\xi|\xi^2\big)\p_\xi \pa+6\big(-2it\s +6 \xi-4 t^2 \xi^2-12it|\xi|\xi^2 +9 \xi^4\big)\p_\xi^2\pa\\
&+4\big(3 \xi^2-2i t |\xi|\big)\p_\xi^3 \pa+\p_\xi^4 \pa\Big)\ka\\
=& \ 16it\mu \p_\xi \hat \phi\delta+A_1+\cdots+A_{28},
\end{split}
\end{equation}
where the terms $A_j$ are such that $A_j=A_j(t,\xi,\hat \phi)$.

Using \eqref{2.9} and \eqref{der4phi} we obtain

\begin{equation}
\begin{split}\label{der4t}
\p_\xi^{4} \hat{u}(t,\xi)=&\underbrace{16i\Big(-t\ps(t,\xi)\p_\xi \pa+\int_0^t (t-\tau)\ps(t-\tau,\xi)\p_\xi \hat{\kappa}d\tau\Big)}_{G(t,\xi)}\delta +\sum_{1\leq j \leq 28} A_j(t,\xi,\hat \phi)\\
&-\sum_{1\leq j \leq 28}\int_0^t A_j(t-\tau,\xi,\hat \kappa(\tau))d\tau\\
=& \ G(t,0)\delta+\sum_{1\leq j \leq 28} A_j(t,\xi,\hat \phi)-\sum_{1\leq j \leq 28}\int_0^t A_j(t-\tau,\xi,\hat \kappa(\tau))d\tau.
\end{split}
\end{equation}

The conservation law \eqref{norm} implies that
\begin{equation}
\p_\xi \hat \kp (\tau,0)=\frac{i}{2}\widehat{u^2}(\tau,0)=\frac{i}{2}\|\phi\|^2.
\end{equation}
Since $\mu(t,0)=\mu(t-\tau,0)=1$, by the last identity we can write
\begin{equation}\label{G}
\begin{split}
 G(t,0)=& \ 16i\Big(-t \p_\xi \hat{\phi}(0)+ \int_0^t (t-\tau)\p_\xi \hat{\kappa}(\tau,0)d\tau\Big)\\
       =& \ -16\Big(\int x \phi(x)dx+\frac12 \int_0^t(t-\tau)\|\phi\|^2 d\tau\Big)\\
       =& -16t\Big(\int x \phi(x)dx+\frac{t}{4}\|\phi\|^2 \Big)\\
       =& \ 0,
\end{split}
\end{equation}
if $$t=t^*=\frac{-4}{\|\phi\|^2}\int x\phi(x)dx.$$

Therefore, putting $t=t^*$ in \eqref{der4t} we obtain

\begin{equation}\label{intt}
\begin{split}
\p_\xi^4(\widehat{u(t^*)})=& \sum_{1\leq j \leq 28} A_j(t^*,\xi,\hat \phi)-\sum_{1\leq j \leq 28}\int_0^{t^*} A_j(t^*-\tau,\xi,\hat \kappa(\tau))d\tau,
\end{split}
\end{equation} 
where the delta Dirac function is not contained in the last identity. In the follow argument we are always assuming $t=t^*$. Thus, taking $D_\xi^\theta$ in both sides of \eqref{intt} and using Plancherel identity 
it's seens that

\begin{equation}\label{inttp}
\begin{split}
(|x|^{4+\theta}u(t))^{\wedge}(\xi)=& \sum_{1\leq j \leq 28} D_\xi^\theta A_j(t,\xi,\hat \phi)-\sum_{1\leq j \leq 28}\int_0^t D_\xi^\theta A_j(t-\tau,\xi,\hat \kappa(\tau))d\tau.
\end{split}
\end{equation}

The next step is to show that all terms on the right-hand side of \eqref{inttp} belongs to $L^2(\R)$.
We will estimate some only terms.

The Lemma \ref{DF} and Proposition \ref{intj} (with $j=2$, $k=4$) imply that 
\begin{equation}
\begin{split}
\|D_\xi^{\theta}A_{25}(t,\cdot,\hat \phi(\cdot))\|\les& \ \|J^{2\theta}\p_x^4(x^2 \phi)\|+\||x|^\theta \p_x^4(x^2 \phi)\|\\
                        \les& \ \|J^{2\theta}\p_x^2 \phi\|+\|J^{2\theta}(x\p_x^3 \phi)\|+\|J^{2\theta}(x^2 \p_x^4 \phi)\|+\||x|^\theta \p_x^4(x^3 \phi)\|\\
                        \les& \ \|J^{2(\theta+1)+3}\phi\|+\|\lanx^{\theta+1+\frac{3}{2}}\phi\|+\|J^{2(\theta+4)}\phi\|+\|\lanx^{4+\theta}\phi\|\\
                        &+\|\lanx^{\frac{2(4+\theta)(2+\theta)}{2(4+\theta)-4}}\phi\|\\
                        \les& \ \|J^{2(4+\theta)}\phi\|+\||x|^{4+\theta}\phi\|.
\end{split}
\end{equation}

In addition, an application of Lemma \ref{DF}, Lemma \ref{boundhilbert} and Proposition \ref{intj} (with $j=3$, $k=1$) gives us

\begin{equation}
\begin{split}
\|D_\xi^{\theta}A_{28}(t,\cdot,\hat \phi(\cdot))\|\les& \ \|J^{2\theta}\h \p_x (x^3 \phi)\|+\||x|^\theta \h \p_x (x^3 \phi)\|\\
                        \les& \ \|J^{2\theta}(x^2 \phi)\|+\|J^{2\theta}(x^3\p_x \phi)\|+\||x|^\theta \p_x(x^3 \phi)\|\\
                        \les& \ \|J^{2(\theta+3)+1}\phi\|+\|\lanx^{\theta+3+\frac{1}{2}}\phi\|+\|J^{2(\theta+4)}\phi\|+\|\lanx^{\frac{2(4+\theta)(3+\theta)}{2(4+\theta)-1}}\phi\|\\
                        \les& \ \|J^{2(4+\theta)}\phi\|+\||x|^{4+\theta}\phi\|.
\end{split}
\end{equation}

Also, by using Proposition \ref{boundhilbert} (with $\omega=|x|^\theta$) we get

\begin{equation*}
\begin{split}\label{A21}
\|D_\xi^{\theta}A_{21}(t,\cdot,\hat \phi(\cdot))\|&\les \|J^{2\theta}\h (x^2 \phi)\|+\||x|^\theta \h(x^2 \phi)\|\\
&\les \|J^{2\theta}(x^2 \phi)\|+\||x|^\theta x^2 \phi\|\\
&\les \|J^{2(\theta+2)}\phi\|+\|\lanx^{2+\theta}\phi\|.
\end{split}
\end{equation*} 

To deal with the other terms $A_j(t,\xi,\hat \phi)$ we can proceed in a similar way, to obtain,
for all $j=1,...,28$ 
\begin{equation}\label{Aj}
\|D_\xi^\theta A_j(t,\cdot,\hat \phi(\cdot))\|\les \|J^{2(4+\theta)}\phi\|+\||x|^{4+\theta}\phi\|.
\end{equation}

About the integral terms in \eqref{inttp}, putting $k=\frac12 \p_x u^2$ instead of $\phi$, in the above estimates, follows that, for any $j=1,...,28$ and $\tau\in [0,t]$

\begin{equation}\label{kpest}
\|D_\xi^\theta A_j(t,\cdot,\hat \kappa(\tau))\|\les \|J^{2(4+\theta)}\kp(\tau)\|+\||x|^{4+\theta}\kp(\tau)\|.
\end{equation}



Next, as in \eqref{j120} follows that
\begin{equation}
\begin{split}\label{j2theta}
\|J^{2(4+\theta)}\kp(\tau)\|\les \|u(\tau)\|_{H^{9+2\theta}}^2 .
\end{split}
\end{equation}

In addition, from Sobolev's embedding and Proposition \ref{jota}, with $\beta=\frac{1}{8+2\theta}$ and $\nu=\frac{(2+\theta)(8+2\theta)}{7+2\theta}$, follow that

\begin{equation}\label{kappa}
\begin{split}
 \||x|^{4+\theta}\kappa(\tau)\|\les & \ \||x|^{4+\theta}u u_x\|\\
                               \les & \ \|x^2 u\|_{L^\infty}\||x|^{2+\theta}u_x\|\\
                               \les & \ \|\lanx^{2+\theta} u\|_{L^\infty}(\|\p_x(|x|^{2+\theta}u)\|+\||x|^{1+\theta}u\|)\\
                               \les & \ \|J(\lanx^{2+\theta}u)\|(\|J(\lanx^{2+\theta}u)\|+\|\lanx^{1+\theta}u\|)\\
                               \les & \ \|J^{2(4+\theta)}u\|^2+\|\lanx^{\nu}u\|^2.
\end{split}
\end{equation}
We observe that the terms above are finite, in view of $\nu<7/2$ and $u\in C(\R;\z_{9+2\theta,7/2-})$.

From \eqref{kpest}, \eqref{j2theta} and \eqref{kappa}

\begin{equation}\label{kappa1}
\begin{split}
\int_0^t \|D_\xi^\theta A_j(t,\cdot,\hat \kappa(\tau))\|d\tau &\les \int_0^t \big(\|J^{2(4+\theta)}u\|^2+\|\lanx^{\nu}u\|^2\big)d\tau\\
&\les \sup_{[0,T]}\|u(t)\|_{\z_{s,\nu}}^2, \quad j=1,...,28.
\end{split}
\end{equation}

Therefore, using \eqref{intt}, \eqref{Aj} and \eqref{kappa1} we conclude that

$$u(t^*) \in \z_{9+2\theta,4+\theta}.$$ 

This finishes the proof. 

\end{proof}

\section{Uniqueness results}\label{uniqueness1}

\begin{proof}[Proof of Theorem \ref{uni1}]
Let $w:=u-v$, then the IVP \eqref{b} implies that 
\begin{equation}\label{bu1}
\p_t w+\h \p_x^2 w+\p_x^3 w+u\p_x w+w\p_x v=0, \quad (x,t)\in \R\times [0,T].
\end{equation}
The hypothesis \eqref{interval} gives us, for all $x\in I$
\begin{equation}\label{bw1}
w(x,0)=\p_x w(x,0)=\p_t w(x,0)=\p_x^3w(x,0)=0.
\end{equation}
By \eqref{bu1} and \eqref{bw1}, we obtain 
\begin{equation}\label{hbw}
\h \p_x^2 w(x,0)=0, \quad \mbox{for any} \  x\in I.
\end{equation}
Using \eqref{bw1} and \eqref{hbw} it follows that
\begin{equation}
\p_x^2w(x,0)=\h\p_x^2 w(x,0)=0, \quad \forall x\in I,
\end{equation}
where $\p_x^2 w(\cdot,0)\in H^s(\R), s>1/2$.

Thus, from Lemma \ref{hzero} we obtain $\p_x^2 w(\cdot,0)\equiv 0$. 

Therefore $w(x,0)=0$, for all $x\in \R$, and by the uniqueness in $H^s(\R)$ it follows that

$$u(x,t)=v(x,t),\quad (x,t)\in \R\times [0,T].$$

This finishes the proof.
\end{proof}

\begin{proof}[Proof of Theorem \ref{uni2}]
Let $w(x,t)=(u-v)(x,t)$, then by the IVP \eqref{b}
\begin{equation}\label{bu2}
\p_t w+\h \p_x^2 w+\p_x^3 w+u\p_x w+w\p_x v=0, \quad (x,t)\in \R\times [0,T].
\end{equation}
From \eqref{indata} $w(x,0)=0$, for any $x\in I$. Then $\p_x^j w(x,0)=0$, for any $x\in I$, where $j=0,...,3$. As a consequence, by \eqref{bu2}
\begin{equation}\label{ptw}
\h \p_x^2 w(x,0)=-\p_t w(x,0), \quad \mbox{for all} \ x \in I.
\end{equation}
Using \eqref{intcon} and \eqref{ptw} we conclude that
\begin{equation}
\begin{split}
\int_{|x|\leq R}|\p_t u(x,0)-\p_t v(x,0)|^2 dx=&\int_{|x|\leq R}|-\p_t w(x,0)|^2dx\\
                                                  =&\int_{|x|\leq R}|\h \p_x^2 w(x,0)|^2 dx\\
                                                  \leq& c_N R^N,                                                 \quad \mbox{as} \quad R\downarrow 0.
\end{split}
\end{equation}
Therefore, by the Lemma \eqref{fzero} 
\begin{equation}
\p_x^2 w(x,0)=0, \quad \mbox{for all} \  x\in \R.
\end{equation}
The rest run as in the proof of Theorem \ref{uni1}.

This ends the proof.
\end{proof}

\section*{Acknowledgment}

The author would like to thank Prof. Ademir Pastor, for his attention in reading this manuscript.

\end{document}